\documentclass[a4paper,12pt,leqno]{amsart}
\usepackage[colorlinks=true,linkcolor=darkblue,citecolor=darkblue]{hyperref}

\usepackage{graphicx}
\usepackage{pinlabel} 
\usepackage{amsmath}
\usepackage{amsfonts}
\usepackage{amssymb}
\usepackage{mathtools}
\usepackage[all,cmtip]{xy}
\usepackage{amsthm}
\usepackage{tikz-cd}
\usepackage{comment}
\usepackage{enumerate}
\usepackage{url}
\usepackage{epsfig}
\usepackage[utf8]{inputenc}

\usepackage[capitalise]{cleveref}
\usepackage{geometry}
\makeatletter
\providecommand\@dotsep{5}
\def\listtodoname{List of Todos}
\def\listoftodos{\@starttoc{tdo}\listtodoname}
\makeatother

\newtheorem{theorem}{Theorem}[section]
\newtheorem{proposition}[theorem]{Proposition}
\newtheorem{corollary}[theorem]{Corollary}
\newtheorem{lemma}[theorem]{Lemma}

\newcommand{\mycomment}[1]{}

  \theoremstyle{definition}
\newtheorem{definition}[theorem]{Definition}

\newtheorem{remark}[theorem]{Remark}

\newcommand{\Z}{\mathbb{Z}}

\newcommand{\nbeq}{\begin{equation}}
\newcommand{\neeq}{\end{equation}}
\newcommand{\beq}{\begin{equation*}}
\newcommand{\eeq}{\end{equation*}}

\DeclareMathOperator{\cd}{cd}
\DeclareMathOperator{\asdim}{asdim}

\DeclareMathOperator{\vcd}{vcd}

\DeclareMathOperator{\asd}{asdim}

\DeclareMathOperator{\gd}{gd}

\DeclareMathOperator{\mcg}{Mod}

\DeclareMathOperator{\pmcg}{PMod}

\DeclareMathOperator{\fps}{FP}
\DeclareMathOperator{\pb}{PB}

\definecolor{darkblue}{rgb}{0.0, 0.0, 0.55}
\usepackage{verbatim}

\begin{document}

\title[]{Asymptotic and cohomological dimension of surface braid groups and poly-surface groups}

\author[Porfirio L. León Álvarez]{Porfirio L. León Álvarez}

\address{Instituto de Matemáticas, Universidad Nacional Autónoma de México. Oaxaca de Juárez, Oaxaca, México 68000}
\email{porfirio.leon@im.unam.mx}

\author[Israel Morales]{Israel Morales}

\address{Departamento de Matemática y Estadística. 
Universidad de La Frontera. Avenida Francisco Salazar 01145, Temuco, Chile.}
\email{israel.morales@ufrontera.cl}


\date{\today}


\keywords{Asymptotic dimension, surface braid groups, poly-free groups, duality groups, virtual cohomological dimension}

\subjclass{20F69, 20F65, 57M07}

\begin{abstract}
In this paper, we determine the asymptotic dimension for all surface braid groups —including those associated with non-orientable and infinite-type surfaces— as well as for torsion-free poly-finitely generated surface groups. We demonstrate that for both classes, the virtual cohomological dimension and the asymptotic dimension coincide. For poly-finitely generated surface groups and braid groups of finite-type surfaces, our approach establishes that these groups are virtual duality groups in the sense of Bieri-Eckmann. In the case of infinite-type surfaces, the argument is based on the fact that their braid groups are countable and normally poly-free.

\end{abstract}

\maketitle
\section{Introduction}
In \cite{Gr93}, M. Gromov introduced the notion of \emph{asymptotic dimension} ($\mathrm{asdim}$) for groups and, more broadly, for metric spaces, as a quasi-isometric invariant (see \cref{asdim:def}). Beyond its intrinsic importance in coarse geometry, asymptotic dimension maintains deep connections with several fundamental conjectures in topology and geometry. Notably, Yu established \cite{YU} the coarse Baum-Connes conjecture for groups with finite asymptotic dimension. Furthermore, the asymptotic dimension provides an upper bound for the nuclear dimension of $C^*$-algebras \cite[Theorem 8.5]{WINTER2010461} and it is closely linked to the concept of dynamical asymptotic dimension \cite[Theorem 1.3]{MR3606454}.

While upper bounds for asymptotic dimension are known for several classes of groups—such as hyperbolic groups  \cite{Gr93} and mapping class groups \cite{Bestvina}—explicit computations remain, in general, a difficult challenge. To date, precise values have been determined for poly-\(\mathbb{Z}\) groups \cite[Theorem 3.5]{A2006}, for right-angled Artin groups (RAAGs) \cite[Theorem 1.2]{MR4665278}, and recently, for \( 3 \)-manifold groups \cite{Haydee}. 

Although it is established that braid groups of finite-type surfaces have finite asymptotic dimension, their exact values are known only in sporadic cases. In this paper, we prove that the asymptotic dimension of surface braid groups coincides with their virtual cohomological dimension, providing an explicit computation for all such groups.

\begin{theorem}\label{main:them:0}
Let $\Sigma$ be a surface of finite type and $n\geq 1$. The asymptotic dimension of the $n$--strand surface braid group $\mathrm{B}_n(\Sigma)$ coincides with its virtual cohomological dimension; that is,
\[
\asd(\mathrm{B}_n(\Sigma)) = \vcd(\mathrm{B}_n(\Sigma)).
\]

\end{theorem}

The proof of \cref{main:them:0} rests on the fact that the braid group of a finite-type surface 
 is a virtual duality group in the sense of Bieri–Eckmann (see \cref{duality:groups}). Notable examples of such groups include mapping class groups of orientable surfaces \cite{harer1986}, certain arithmetic groups such as  \cite{BorelSerre1973}, and, as recently established, handlebody groups \cite{PetersenWade2024}.
 
Although the $\vcd$ of braid groups over closed surfaces is well-documented (see \cref{Rmk:VCD-BraidGroups}), we offer here a systematic and self-contained computation applicable to the entire class of finite-type surface braid groups. For $g\in \Z$ and $r,k \geq 0$, if $g$ is nonnegative (resp., $g$ is negative), $\Sigma_{g,k}^r$ denotes the surface of genus $\vert g \vert$ with $r$ boundary components and $k$ punctures, which is orientable if $g\geq 0$ and non-orientable if $g<0$.

\begin{theorem}\label{thm:duality:braid:0}
Let $\Sigma = \Sigma_{g,k}^r$ be a surface of finite type and $n \geq 1$. Then the $n$--strand surface braid group $\mathrm{B}_n(\Sigma)$ is a virtual duality group with virtual cohomological dimension given by:
\[
\vcd\bigl(\mathrm{B}_n(\Sigma)\bigr) =
\begin{cases}
    n - 3 & \text{if } g = 0,\ k + r = 0,\ n \geq 4, \\
    n - 2 & \text{if } g = -1,\ k + r = 0,\ n \geq 3, \\
    n + 1 & \text{if } (g = 1 \mbox{ or } |g| \geq 2) \mbox{ and } \ k + r = 0, \\
    n     & \text{if } (g = 0 \mbox{ and } \ k + r \geq 2) \mbox{ or } (|g| \geq 1 \mbox{ and }\ k + r \geq 1),\\
    n - 1 & \text{if } g = 0,\ k + r = 1, \\
    0     & \text{otherwise.}
\end{cases}
\]
Moreover, if $\Sigma$ is neither the sphere ($g=0,k+r=0$) nor the projective plane ($g=-1, k+r=0$), then $\mathrm{B}_n(\Sigma)$ is a duality group.
\end{theorem}

    

\begin{remark}\label{Rmk:VCD-BraidGroups}


The virtual cohomological dimension of braid groups of closed surfaces has been previously established in the literature (see \cite{maldonado:Guaschi:1,MR3869010}). For the closed disk, the $n$-strand braid group coincides with the classical Artin braid group $\mathrm{B}_n$, which is isomorphic to the mapping class groups of the closed disk with $n$ punctures in its interior \cite{MR0375281}. The cohomological dimension in this case is well known due to the work of Harer \cite{harer1986}, see also \cite[Section 3]{Arnold2014}. 
\end{remark}

Recently, there has been a growing interest in the mapping class groups of infinite-type surfaces—frequently termed \emph{``big mapping class groups''} in the literature \cite{Aramayona-Vlamis2020}. Unlike their finite-type counterparts, big mapping class groups are uncountable and fail to be compactly generated. In contrast, braid groups of infinite-type surfaces are torsion-free, countable, and infinitely generated (see \cref{braid:groups:infinite:type:are normally:polifree}). For such non-finitely generated countable groups, Dranishnikov and Smith \cite{A2006} defined the asymptotic dimension as the supremum of the asymptotic dimensions of its finitely generated subgroups, showing that this extension remains a quasi-isometric invariant.

In this paper, we prove that the asymptotic dimension of the $n$--strand braid group of an infinite-type surface is exactly $n$, coinciding with both its \emph{cohomological dimension} ($\cd$) and \emph{geometric dimension} ($\gd$) (see \cref{cd:gd:def}). We note that these groups are not duality groups, as the latter are necessarily finitely generated. Our argument builds on our previous work on finite-type surface braid groups and the observation that the $n$--strand pure braid group of an infinite-type surface is normally poly-free of length $n$ (see \cref{Subsection:Poly-SurfaceGroups} and \cref{braid:groups:infinite:type:are normally:polifree}).

\begin{theorem}\label{asdim:infinite:type}
 Let $\Sigma$ be a surface of infinite type and $n\geq 1$. Then the $n$--strand pure surface braid group $\mathrm{PB}_n(\Sigma)$ satisfies: $$\asd(\pb_n(\Sigma))=\cd(\pb_n(\Sigma))=\mathrm{gd}(\pb_n(\Sigma))=n.$$    
\end{theorem}

\subsection*{Asymptotic dimension of torsion-free poly-surface groups} We provide a formula for the asymptotic dimension of torsion-free, poly-finitely generated (poly-f.g.) surface groups. The core of the argument rests on demonstrating that this class consists of duality groups in the sense of Bieri–Eckmann \cite{Eckmann1973}. We refer the reader to \cref{Thm:GeneralResult} for a broader class of groups where the same conclusion holds.

\begin{theorem}[Asymptotic dimension of poly-f.g.-surface groups]\label{asim:dim:poly:surface:0}
Let \( G \) be a torsion-free poly-finitely generated surface group of length \( n \), with filtration \(
1 = G_0 \lhd G_1 \lhd \cdots \lhd G_{n-1} \lhd G_n = G\). Then $G$ is a duality group of dimension \[\asd(G) = \cd(G)=\sum_{i=0}^{n-1} \asd\left( \frac{G_{i+1}}{G_i} \right).\]
\end{theorem}

\begin{remark}
By \cref{asim:dim:poly:surface:0}, if $G$ is a poly-f.g. free group of length $n$, then $\asd(G) = \cd(G) = n$. This implies that every filtration of $G$ must have length $n$ (see \cref{asim:dim:poly:free}). While it was shown in \cite[Theorem 1.1]{jimenezleon} that the cohomological dimension and the length of a normally poly-f.g. free group coincide, we provide here an independent proof. We note that poly-\(\mathbb{Z}\) groups also satisfy this property \cite[Example 5.26]{Lu05}. Furthermore, the fact that any two filtrations of a poly-f.g. free group share the same length \cite[Theorem 16]{MR596323} can also be recovered from \cref{asim:dim:poly:surface:0}.
\end{remark}

\subsection*{Outline of the paper}
In \cref{Preliminaries}, we review the necessary background and preliminary results. \Cref{asym:poly:surfaces} is devoted to proving that the asymptotic dimension of torsion-free poly-f.g. surface groups coincides with their cohomological dimension. In \cref{basics:braid:groups}, we establish the same identity for braid groups of finite-type surfaces and provide an explicit computation of their virtual cohomological dimension. Finally, \cref{Secc:cd-SurfacesIT} addresses braid groups of infinite-type surfaces, showing that their asymptotic dimension also matches their virtual cohomological dimension.\\

\noindent{\bf Acknowledgments.} The authors are grateful to Jesús Hernández Hernández, Rita Jiménez Rolland, Daniel Juan Pineda, and Luis Jorge Sánchez Saldaña for their helpful discussions and insights. The first author would like to thank Noé Bárcenas and Carlos Pérez Estrada for the invitation to the seminar “Prospectos en topología” at the Center of Mathematical Sciences (CCM-UNAM, Morelia), where the initial ideas regarding asymptotic dimension for this work were discussed. The first author was supported by the UNAM Postdoctoral Program (POSDOC), and the second author was supported by ANID Postdoctoral Fellowship 3240229.

\section{Preliminaries}\label{Preliminaries}
In this section, we collect the fundamental definitions and results required for the proofs of our main theorems.

\subsection{Asymptotic dimension of groups}\label{asdim:def} The \emph{asymptotic dimension} of a metric space $(X, d)$, denoted by $\asd(X)$, is the smallest integer $n$ such that for every $R > 0$, there exists a cover $\mathcal{U}$ of $X$ by sets of have uniformly bounded diameter such that every ball of radius $R\geq 0$ in $X$ intersects at most $n+1$ elements of $\mathcal{U}$.

For a finitely generated group $G$, any two Cayley graphs associated with different finite generating sets are quasi-isometric. Since asymptotic dimension is a quasi-isometry invariant, the following definition is well-posed:


\begin{definition}[Asymptotic dimension of finitely generated groups]
  The asymptotic dimension of a finitely generated group $G$, denoted by $\asd(G)$, is the asymptotic dimension of any of its Cayley graphs equipped with the word metric.   
\end{definition}

Following \cite{A2006}, this notion extends naturally to the class of countable groups:


\begin{definition}
The asymptotic dimension of a countable group $G$ is defined as
\[
\asd(G) := \sup \{\asd(H) \mid H\leq G \text{ is a finitely generated subgroup}\}.
\]   
\end{definition}

The following result, which relates the asymptotic dimension of groups within a short exact sequence, will be fundamental for our computations.


\begin{theorem}[Hurewicz type formula]\cite[Theorem 2.3]{A2006}\label{HurewiczFormula}
    Let $1\to K\to G \to H\to 1$ be a short exact sequence of discrete groups. Then $$\asdim(G)\leq \asdim(K)+\asdim(H).$$
\end{theorem}

\subsection{Cohomological and geometric dimension}\label{cd:gd:def}
In this subsection, we recall the definitions of the geometric and cohomological dimensions of a discrete group; for a comprehensive treatment, we refer the reader to ~\cite{Br94}.


A model for the \emph{classifying space} $BG$ (or $\mathrm{K}(G,1)$) of a discrete group $G$ is a connected CW-complex $X$ whose fundamental group is isomorphic to $G$ and whose universal cover
is contractible. Such a space always exists and is unique up to homotopy equivalence.

The \emph{geometric dimension} of $G$, denoted $\gd(G)$, is the minimum dimension of
a CW-complex model for $BG$. Its algebraic analogue is the \emph{cohomological
dimension} $\cd(G)$, defined as the projective dimension of the trivial $\mathbb{Z}G$-module $\mathbb{Z}$:
\[
  \cd(G) \coloneqq \min\bigl\{ n \mid \mathbb{Z} \text{ admits a projective resolution
  of length } n \text{ over $\mathbb{Z}G$} \bigr\}.
\]

These invariants are related by the \emph{Eilenberg--Ganea Theorem}, which states that:
\[
  \cd(G) \leq \gd(G) \leq \max\{3,\, \cd(G)\}.
\]

The question of whether $\gd(G) = \cd(G)$ always holds
is the celebrated \emph{Eilenberg--Ganea conjecture}, which remains open in the
case $\cd(G) = 2$.

If $G$ contains an element of finite order, then $\cd(G) = \infty$. To address groups with torsion, we use the \emph{virtual cohomological dimension} $\vcd$. If $G$ is virtually torsion-free, its $\vcd$ is defined as
\[
  \vcd(G) \coloneqq \cd(H),
\]
where $H \leq G$ is any torsion-free subgroup of finite index. A theorem of Serre
ensures that this value is independent of the choice of $H$, making $\vcd(G)$ well-defined.

\subsection{Surface braid groups}
Surface braid groups arise naturally as the fundamental groups of configuration spaces of surfaces, generalizing Artin's classical braid groups \cite{kassel2008}. Beyond their intrinsic interest in geometric group theory and low-dimensional topology, they play a crucial role in various areas, such as the study of mapping class groups via the Birman exact sequence \cite{FM12}. For a comprehensive survey of their properties, we refer the reader to \cite{MR3382024}. Below, we present the general framework applicable to all topological surfaces.

\subsubsection{Topological surfaces} 
Throughout this paper, unless otherwise specified, a surface is a connected topological surface with (possibly empty) compact boundary, either orientable or non-orientable. We denote the \emph{boundary} and \emph{interior} of $\Sigma$ by $\partial \Sigma$ and $\mathrm{int}(\Sigma)$, respectively.

A surface $\Sigma$ is of \emph{finite type} if its fundamental group is finitely generated; otherwise, it is of \emph{infinite type}. By the Classification Theorem for surfaces \cite{richards1963}, any finite-type surface is homeomorphic to some $\Sigma_{g,k}^r$. We omit $r$ or $k$ from the notation when they are zero.

\subsubsection{Surface braid groups} 
Let $\Sigma$ be a surface and $n \geq 1$. The $n$-th \emph{configuration space} of $\Sigma$ is the set \[F_n(\Sigma) := \{ (x_1, x_2, \ldots, x_n) \in (\mathrm{int}(\Sigma))^n \mid x_i \neq x_j \text{ for } i \neq j \}
\] endowed with the subspace topology from $\Sigma^n$. Under this topology, $F_n(\Sigma)$ is an open connected manifold of dimension $2n$ which is aspherical provided $\Sigma \neq \mathbb{S}^2,\mathbb{RP}^2$. 

The \emph{pure braid group} on $n$ strands, denoted by $\pb_n(\Sigma)$, is defined as the fundamental group $\pi_1(F_n(\Sigma))$. The symmetric group $\frak{S}_n$ acts freely on $F_n(\Sigma)$ by permuting coordinates and the \emph{braid group} on $n$ strands, $\mathrm{B}_n(\Sigma)$, is defined as the fundamental group of the quotient space $C_n(\Sigma):=F_n(\Sigma)/\frak{S}_n$. For $\Sigma \neq \mathbb{S}^2,\mathbb{RP}^2$, both $\pb_n(\Sigma)$ and $\mathrm{B}_n(\Sigma)$ are torsion-free \cite[Section 1.4]{kassel2008} and are related by the short exact sequence: 
\begin{equation}
1 \to \pb_n(\Sigma) \to \mathrm{B}_n(\Sigma) \to \frak{S}_n \to 1.    
\end{equation}

\begin{remark}
    Since the asymptotic dimension is a quasi-isometry invariant and $\pb_n(\Sigma)$ is a finite-index subgroup of $\mathrm{B}_n(\Sigma)$, it follows that $\asd(\pb_n(\Sigma))=\asd(\mathrm{B}_n(\Sigma)).$
\end{remark}


\subsubsection{The Fadell--Neuwirth short exact sequence} 
For a surface \( \Sigma \) with $\partial\Sigma = \emptyset$, the \emph{forgetting map} $\rho: F_{n+1}(\Sigma) \to  F_n(\Sigma)$ is a fibration with fiber homeomorphic to \( \Sigma - Q_n \), where \( Q_n\) is a set of \( n \) distinct points in \( \Sigma \) (see \cite{FadellNeu} or \cite[Lemma 1.27, Exercise 1.4.1 ]{kassel2008}). This implies that \( F_n(\Sigma) \) is an Eilenberg--MacLane space for $n\geq 1$ (excluding $\mathbb{S}^2$ and $\mathbb{RP}^2$). Consequently, the long exact sequence in homotopy yields: 

\begin{proposition}[Fadell-Neuwirth short exact sequence of surface braid groups]
\label{Fadell:Neuwirth}
Let \( \Sigma \) be a surface with empty boundary, other that $\mathbb{S}^2$ and $\mathbb{RP}^2$. For all $n\geq 1$, there is a short exact sequence
\[
1 \to \pi_1(\Sigma \setminus Q_n) \to \pb_{n+1}(\Sigma) \xrightarrow{\varphi} \pb_n(\Sigma) \to 1.
\]
\end{proposition}

It suffices to compute the asymptotic dimension for braid groups of surfaces with empty boundary. Indeed, capping each boundary component of $\Sigma_{g,k}^r$ with an once-punctured disk induces a homotopy equivalence $i\colon \Sigma_{g,k}^r \to \Sigma_{g,k+r}$, which in turn induces a homotopy equivalence between their respective configuration spaces.

\begin{lemma}\cite[Remark 8 item (d)]{MR3382024}\label{surfacebraid:puntures:boundaries}
    Let $\Sigma_{g,k}^r$ be a finite-type surface and $n\geq 1$. Then $\pb_n(\Sigma_{g,k}^r) \cong \pb_n(\Sigma_{g,k+r})$.
\end{lemma}

\subsection{Mapping class groups of finite-type surfaces} For a comprehensive introduction to mapping class groups, we refer the reader to \cite{FM12}. The {\it mapping class group} of finite-type surface $\Sigma$, denoted by $\mcg(\Sigma)$, is the group of all isotopy classes of homeomorphisms $f\colon \Sigma \rightarrow \Sigma$ that restrict to the identity on $\partial \Sigma$. If $\Sigma$ is orientable, these homeomorphisms are further required to be orientation-preserving. The \emph{pure mapping class group} of $\Sigma$, $\pmcg(\Sigma)$, is the subgroup consisting of mapping classes that fix each puncture individually. For surface with $n$ punctures, these groups are related by the  short exact sequence:  
\begin{equation}
    1 \to \pmcg(\Sigma) \to \mcg(\Sigma) \to  \frak{S}_n \to 1.
\end{equation}

For any surface of finite type, $\mcg(\Sigma)$ is finitely presented (see \cite[Section 5.2]{FM12} for the orientable case and \cite{Szepietowski2008} for the non-orientable case).\\    

Our proofs rely on the Birman exact sequence for surfaces of finite type. While the orientable case is established in \cite[Theorem 4.6]{FM12}, the argument remains valid in the non-orientable setting. 

\begin{theorem}[Birman exact sequence]\label{thm:BES}
    Let $\Sigma_{g,k}^r$ be a finite-type surface with negative Euler characteristic. Then the following sequence is exact:

    \begin{equation*}\label{eq:SEB}
    1\rightarrow\pi_1(\Sigma_{g,k}^r)\rightarrow \pmcg(\Sigma_{g,k+1}^r)\rightarrow \pmcg(\Sigma_{g,k}^r) \rightarrow 1.
    \end{equation*}
\end{theorem}

The following lemma is used in \cref{asd:sphere} to compute the $\vcd$ of the braid groups of the sphere and the projective plane.

\begin{lemma}\label{asd:mcg:Birman}
Let $\Sigma_{g,k}$ be a finite-type surface with $\partial \Sigma = \emptyset$ and negative Euler characteristic. If $\asd(\mcg(\Sigma_{g,k}))=\vcd(\mcg(\Sigma_{g,k}))$ for some $k\geq 0$ then for all $n\geq k$ \[\asd(\mcg(\Sigma_{g,n}))=\vcd(\mcg(\Sigma_{g,n})).\]
 \end{lemma}

\begin{proof}
We proceed by induction on $n\geq k$. The base case $n=k$ holds by hypothesis. Assuming the result for $n$ we consider the Birman exact sequence (\cref{thm:BES})  
\begin{equation}\label{eq:21}
1\rightarrow\pi_1(\Sigma_{g,n})\rightarrow \pmcg(\Sigma_{g,n+1})\rightarrow \pmcg(\Sigma_{g,n}) \rightarrow 1.
\end{equation} 

By \cref{surfaces:groups:duality:groups} and \cref{asd:surfaces}, $\pi_1(\Sigma_{g,n})$ is a duality group of dimension $\asd(\pi_1(\Sigma_{g,n}))=1$. Furthermore, $\pmcg(\Sigma_{g,n})$ is a virtual duality group (see \cite[Theorem 4.1]{harer1986} and \cite[Theorem 6.9]{I87}) which, by the inductive hypothesis, satisfies $\asd(\pmcg(\Sigma_{g,n}))=\vcd(\pmcg(\Sigma_{g,n}))$. The result then follows from \cref{Thm:GeneralResult}. 
\end{proof}

\subsection{(Virtual) duality groups}\label{duality:groups}
We briefly recall the notion of duality groups in the sense of Bieri--Eckmann and state several results required for our proofs. For a detailed treatment, we refer the reader to \cite{Eckmann1973}. Duality groups generalize Poincaré duality groups by extending their fundamental properties into the framework of group (co)homology \cite{Br94}.\\

Given a group $G$ and a right $G$-module $C$, $H_k(G;C)$ and $H^k(G;C)$ denotes the \emph{$k$-th homology and cohomology} group of $G$ with coefficients in $C$, respectively. 

\begin{definition}
    A group $G$ is a \emph{duality group} of dimension $n$ with respect to the dualizing module $C$ if there exists a fundamental class $e \in H_n(G; C)$ such that the cap-product with $e$ induces isomorphisms 

\[
(e \cap -) \colon H^k(G; A) \xrightarrow{\cong} H_{n-k}(G; C \otimes A),
\]
for all $k\in \Z$ and all left $G$-modules $A$. A group $G$ is a \emph{virtual duality group}  of dimension $n$ if it contains a finite-index subgroup that is a duality group of dimension $n$.
\end{definition}

For any (virtual) duality group, its dimension coincides with its (virtual) cohomological dimension. We state this for future reference.

\begin{proposition}\cite[Proposition 1.2]{Eckmann1973}\label{remark:cd-duality-group}
If $G$ is a duality group (resp., virtual duality group) of dimension $n$ then $\cd(G)=n$ (resp., $\vcd(G)=n$).
\end{proposition} 

We now recall two results from \cite{Eckmann1973} concerning extensions and subgroups.

\begin{theorem}\cite[Theorem 3.3 and Theorem 2.2]{Eckmann1973}\label{Duality:group:extension}
  Let $G$ be a torsion-free group and  $H\leq G$ a finite-index subgroup. Then  $G$ is a duality group of dimension $n$ if and only if  $H$ is a duality group of dimension $n$.  
 \end{theorem}
 
The following result generalizes \cite[Theorem 3.5]{Eckmann1973} to the virtual case.

\begin{theorem}\label{Duality:group:s.e.s}
 Let $1\to K\to G\xrightarrow[]{p} H\to 1$ be a short exact sequence of groups. If $K$ is a duality group of dimension $n$ and $H$ is a duality (resp., virtual duality) group of dimension $m$, then $G$ is a duality (resp., virtual duality) group of dimension $n+m$.   
\end{theorem} 

\begin{proof}
The case where $H$ is a duality group is proved in \cite[Theorem 3.5]{Eckmann1973}. Suppose $H$ is a virtual duality group and let $H^\prime \leq H$ be a duality subgroup of finite-index. Since $p^{-1}(H^\prime)$ is a finite-index subgroup of $G$, it suffices to show it is a duality group of dimension $n+m$. Consider the restricted sequence: 
\begin{equation}\label{Equation0}
    1\to K\cap p^{-1}(H^\prime)\to p^{-1}(H^\prime)\to H'\to 1. 
\end{equation}

We have that $K$ is a duality group of dimension $n$ which is torsion free by \cref{remark:cd-duality-group}, and $K\cap p^{-1}(H^\prime)$ is a finite-index subgroup of $K$. By \cref{Duality:group:extension} we conclude that $K\cap p^{-1}(H^\prime)$ is a duality group of dimension $n$. Applying the aforementioned result to \cref{Equation0}, we conclude that $p^{-1}(H^\prime)$ is a duality group of dimension $n+m$.
 \end{proof}

Recall that a group \( G \) is of \emph{type \( \mathrm{FP} \)} over \( \mathbb{Z} \) if the trivial \( G \)-module \( \mathbb{Z} \) admits a finite projective resolution of finitely generated $G$-modules. If \( G \) is a finitely presented group then \( G \) is of type \( \fps \) if and only if there exists a finitely dominated \( K(G,1) \) \cite[Proposition 6.3 and Theorem 7.1 p. 205]{Br94}. Throughout this paper, groups of type $\fps$ are assumed to be finitely presented.

\begin{theorem}\cite[Theorem 10.1]{Br94}\cite[Corollary 1]{brown1975homological}\label{thm:DualitythenFP}
    Every duality group is of type $\fps$.
\end{theorem}

The following result will be an important piece in our arguments.

\begin{theorem}\label{lemma:asdim:dual}
Let $G$ be a virtual duality group of dimension $n$. Then
\[
\asd(G) \geq \vcd(G)=n.
\]
\end{theorem}

\begin{proof}
Let $H\leq G$ be a finite-index duality subgroup of dimension $n$. Then $H$ is of type FP by \cref{thm:DualitythenFP}. From \cite[Theorem 5.11]{Dra:Alex}, we have \[\asd(G)=\asd(H) \geq \cd(H)=\vcd(G)=n.\]
\end{proof}

In the next section we will use the following result in order to stablish that braid groups of a finite-type surface are finitely-generated poly-surface groups (for its definition see \cref{Subsection:Poly-SurfaceGroups}).

\begin{lemma}\label{normmally:polisurfaces:ses}
    Let $1\to K\to G\xrightarrow[]{p} H\to 1$ be a short exact sequence. If $K$ and $H$ are poly-surface groups of length $m$ and $n$, respectively, then $G$ is a poly-surface group of length at most $m+n$. Furthermore, if $H$ is normally poly-surface and $K$ has length one, then $G$ is a normally poly-surface of length at most $n+1$.
\end{lemma}

\begin{proof}
Given a filtration $1 = H_0 \lhd H_1 \lhd \cdots \lhd H_{n-1} \lhd H_n = H$, the preimages $p^{-1}(H_i)$ provide a filtration for $G$ \cite[Theorem 5.11]{MR600654}    
\begin{equation}\label{equ:1:nor}
    \ker(p)= p^{-1}(H_0) \lhd p^{-1}(H_1) \lhd \cdots \lhd p^{-1}(H_{n-1}) \lhd G_n = G.
\end{equation}

Since $p^{-1}(H_i)/p^{-1}(H_{i-1}) \cong H_i/H_{i-1}$ for every $i=1,\ldots, n$ \cite[Corollary 5.8]{MR600654} and $K\cong \ker(p)$ is poly-surface, we can concatenate their filtrations to obtain one for $G$ of length $n+m$. Normality follows if the $H_i$ are normal in $H$ and $K$ is a surface group.


\end{proof}

The next result is used in the computation of the asymptotic dimension of braid groups of the sphere and the projective plane (see Theorem \ref{asd:sphere}).

\begin{lemma}\label{vcd:extension}
Let $1\to F\to G\xrightarrow[]{p} H\to 1$ be a short exact sequence whit $F$ finite. If $G$ is virtually torsion-free then $\vcd(G)=\vcd(H)$ and $\asd(G)=\asd(H)$. Furthermore, if $H$ is a virtual duality group, so is $G$.    
\end{lemma}

\begin{proof}
 Let $L$ be a torsion-free subgroup of $G$ of finite index. As $F$ is finite and $L$ is torsion-free, the restriction of the projection $p\vert_{L}: L\to p(L)$ is an isomorphism. Moreover, $p(L)$ is a finite-index subgroup of $H$. This establishes the first part, as $\vcd(G)=\cd(L)=\cd(p(L))=\vcd(H)$ and similarly for the asymptotic dimension. 

 For the second part, suppose $H$ is a virtual duality group and let $M\leq H$ be a duality subgroup of finite index. The intersection $p(L) \cap M$ is a finite-index subgroup of $M$, and thus it is a duality group by \cref{Duality:group:extension}. Since $p(L) \cap M$ is also a finite-index subgroup of the torsion-free group $p(L)$, applying \cref{Duality:group:extension} again implies that $p(L)\cong L$ is a duality group. Hence, $G$ is a virtual duality group.
\end{proof}

\subsection{Poly-surface groups}\label{Subsection:Poly-SurfaceGroups}

A group \( G \) is called a \emph{poly-surface group} (resp., \emph{poly-free group}) if there exists a finite filtration  
\[
1 = G_0 \lhd G_1 \lhd \cdots \lhd G_{n-1} \lhd G_n = G
\]
such that each quotient \( G_{i+1}/G_i \) is isomorphic to the fundamental group of a surface (resp., to a free group). If each $G_i$ is normal in $G$, we say that $G$ is a \emph{normally poly-surface group} (resp., \emph{normally poly-free group}). Furthermore, if each quotient \( G_{i+1}/G_i \) is the fundamental group of a finite-type surface (resp., a free group of finite rank), then \( G \) is a \emph{poly-finitely generated (poly-f.g.) surface group} (resp., \emph{poly-f.g. free group}). The \emph{length of \( G \)} is the minimal \( n \in \mathbb{N} \) for which such a filtration exists.\\

Typical examples include mapping tori of surfaces, poly-$\mathbb{Z}$ groups, and pure braid groups of surfaces with at least one puncture or non-empty boundary (see \cite{MR1797585} and \cref{braid:groups:infinite:type:are normally:polifree}).

It is a well-known fact that surface braid groups of finite type are poly-f.g. surface groups. For the sake of completeness, we provide a sketch of the proof.

\begin{proposition}\label{purebraid:are:polysurface}
Let \(\Sigma\) be a finite-type surface other than $\mathbb{S}^2$ and $\mathbb{RP}^2$. 

\begin{enumerate}
    \item If $\pi_1(\Sigma)$ is infinite, then for all $n\geq 1$, \(\pb_n(\Sigma)\) is a normally poly-surface group of length at most \(n\). 
    \item If $\Sigma$ is the plane $\mathbb{R}^2$ (the one-punctured sphere $\Sigma_{0,1}$) or the closed disk $\mathbb{D}^2$ (the surface $\Sigma_{0}^1$), then for all $n\geq 2$, \(\pb_n(\Sigma)\) is a normally poly-free group of length at most \(n-1\). 
\end{enumerate}
\end{proposition}

\begin{proof}
By \cref{surfacebraid:puntures:boundaries}, it suffices to consider the case where $\partial \Sigma = \emptyset$. 

For (1), we proceed by induction on $n$. The case $n=1$ is trivial as $\pb_1(\Sigma)=\pi_1(\Sigma)$. Assuming the claim holds for $n=k$, we consider the Fadell--Neuwirth exact sequence (\cref{Fadell:Neuwirth}):
\[
    1 \to \pi_1(\Sigma \setminus Q_k) \to \pb_{k+1}(\Sigma) \to \pb_k(\Sigma) \to 1.
    \]
Since $\pi_1(\Sigma \setminus Q_k)$ is a surface group of length one, by the inductive hypothesis and \cref{normmally:polisurfaces:ses}, $\pb_k(\Sigma)$ is a normally poly-surface group of length at most $k$.

For (2), let $\Sigma$ be the plane. For $n=2$, the Fadell-Neuwirth short exact sequence shows $\pb_{2}(\Sigma) \cong \pi_1(\Sigma \setminus Q_1)$ which is a free group (hence poly-free of length one). The inductive step follows the same logic as (1). 
\end{proof}

\section{Asymptotic dimension of poly-surfaces groups}\label{asym:poly:surfaces}
In this section, we compute the asymptotic dimension of torsion-free, poly-finitely generated (poly-f.g.) surface groups. The upper bound is established via induction on the length of the filtration, while the lower bound relies on showing that these groups are duality groups in the sense of Bieri--Eckmann.

The following results are well-known; we include their proofs for completeness and future reference.

\begin{lemma}\label{surfaces:groups:duality:groups}
Let \( \Sigma \) be a finite-type surface. If $\Sigma$ is the projective plane $\mathbb{RP}^2$, then \( \pi_1(\Sigma) \) is a virtual duality group of dimension $0$. Otherwise, \( \pi_1(\Sigma) \) is a duality group of dimension $m$, where
    \[
        m= 
        \begin{cases}
            2 & \text{if } \Sigma \text{ is closed and } \Sigma \neq \mathbb{S}^2; \\
            0 & \text{if } \Sigma \in \{\mathbb{S}^2, \mathbb{D}^2, \mathbb{R}^2\};\\
            1 & \text{otherwise}.
        \end{cases}
    \]
In particular $\vcd\pi_1(\Sigma))=m$.
\end{lemma}

\begin{proof}
    If $\pi_1(\Sigma)$ if finite, $\Sigma$ is either $\mathbb{RP}^2, \mathbb{S}^2, \mathbb{D}^2$ or $\mathbb{R}^2$. The fundamental group of $\mathbb{RP}^2$ is not trivial, making it a virtual duality group of dimension zero. In the other cases, $\pi_1(\Sigma)$ is trivial, hence a duality group of dimension zero. 

    Now, suppose $\pi_1(\Sigma)$ is infinite, which implies it is torsion-free. If $\Sigma$ is non-orientable, its orientable double-cover $\widetilde{\Sigma} \to \Sigma$ correspond to a subgroup $\pi_1(\widetilde{\Sigma}) \leq \pi_1(\Sigma)$ of index two. By \cref{Duality:group:extension}, it suffices to consider orientable surfaces. If $\Sigma$ is closed, $\pi_1(\Sigma)$ is a duality group of dimension $2$ by Poincaré duality. If $\Sigma$ is not closed, $\pi_1(\Sigma)$ is a free group of finite rank, and thus a duality group of dimension one \cite[Proposition 5.1]{Eckmann1973}.
\end{proof}

\begin{lemma}\label{asd:surfaces}
For any surface $\Sigma$, $\asd(\pi_1(\Sigma))=\vcd(\pi_1(\Sigma))$.
\end{lemma}

\begin{proof}
   Is \( \Sigma \) is of finite type and \( \pi_1(\Sigma) \) is finite, the claim is trivial. If \( \pi_1(\Sigma) \) is infinite and \( \Sigma \) is closed, \( \pi_1(\Sigma) \) is quasi-isometric to its universal cover, \( \mathbb{H}^2 \) or $\mathbb{R}^2$, via the Švarc--Milnor lemma; thus  \( \asd(\pi_1(\Sigma))=2 \). If \( \Sigma \) is of finite-type and not closed, \( \pi_1(\Sigma) \) is free and quasi-isometric to a tree, so \( \asd(\pi_1(\Sigma))=1 \). The result for finite-type surfaces then follows from \cref{surfaces:groups:duality:groups}.
   
   If \( \Sigma \) is of infinite type, \( \pi_1(\Sigma) \) is a free group of countable rank. So, every finitely generated subgroup \( \pi_1(\Sigma) \) is a free group of finite rank, and therefore we have $\asd(\pi_1(\Sigma)) = 1$. Similarly, $\cd(\pi_1(\Sigma))=1$ \cite[p. 185]{Br94}.
\end{proof}

\begin{theorem}\label{Thm:GeneralResult}
    Consider a short exact sequence $1\to K\to G \to H\to 1$. Suppose that $K$ is a duality group and $H$ is a virtual duality group, and both satisfy $\asdim=\vcd$. Then $G$ is a virtual duality group satisfying $\vcd(G)=\asdim(G)$.
\end{theorem}

\begin{proof}
    By \cref{Duality:group:s.e.s} and \cref{remark:cd-duality-group}, $G$ is a virtual duality group with $\vcd(K)+\vcd(H)=\vcd(G)$. The result follows from the chain of inequalities: \[\vcd(K)+\vcd(H)=\asdim(K)+\asdim(H)\geq \asdim(G) \geq \vcd(G).\]

    The first inequality is the Hurewicz-type formula (\cref{HurewiczFormula}), and the second follows from \cref{lemma:asdim:dual}.   
\end{proof}


\noindent{\emph{Proof of \cref{asim:dim:poly:surface:0}}}. We proceed by induction on \( n \). For \( n = 1 \), \( G \cong \pi_1(\Sigma)\) for a torsion-free finite-type surface $\Sigma$, and the result holds by \cref{surfaces:groups:duality:groups} and \cref{asd:surfaces}.

Assuming the result for \( n \), suppose \( G \) has length \( n+1 \) and consider the short exact sequence $1 \to G_n \to G \to G_{n+1}/G_n \to 1$. By the inductive hypothesis, \( G_n \) is a duality group with \( \asd(G_n)= \cd(G_n)\). Since \( G_{n+1}/G_n \) is a virtual duality group satisfying the same identity (again, by \cref{surfaces:groups:duality:groups} and \cref{asd:surfaces}), \cref{Thm:GeneralResult} implies $G$ is a virtual duality group with the required dimension. Since $G$ is torsion free, it is a duality group by \cref{Duality:group:extension}.\qed

\medskip

Given poly-free groups are torsion-free, in this context we can formulate \cref{asim:dim:poly:surface:0} as follows:

\begin{corollary}[Asymptotic dimension of poly-free groups]\label{asim:dim:poly:free}
    If $G$ is a poly-f.g.-free group of length $n$, then $\asd(G) = \cd(G) = n$.
    Consequently, every such filtration of $G$ has length $n$.
\end{corollary}


\section{Asymptotic dimension of finite-type surface braid groups}\label{basics:braid:groups}
This section is dedicated to the proofs of Theorems \ref{main:them:0} and \ref{thm:duality:braid:0}. We first establish that surface braid groups are (virtual) duality groups whose asymptotic dimension coincides with their (virtual) cohomological dimension; these results are presented in Theorems \ref{main:thm}, \ref{main:thm2}, and \ref{asd:sphere}. Subsequently, we provide an explicit computation of these dimensions for all finite-type surfaces.

\begin{theorem}\label{main:thm}
Let \(\Sigma\) be a finite-type surface other than $\mathbb{S}^2$ and $\mathbb{RP}^2$. For all $n\geq 1$, $\mathrm{B}_n(\Sigma)$ is a duality group and \[\asd(\mathrm{B}_n(\Sigma))=\cd(\mathrm{B}_n(\Sigma)).\]  
\end{theorem}

\begin{proof}
Since \( \mathrm{B}_n(\Sigma) \) is torsion-free for any surface \( \Sigma \notin \{\mathbb{S}^2, \mathbb{RP}^2 \}\), the result follows directly from the poly-surface structure established in \cref{purebraid:are:polysurface} combined with the dimension formula in \cref{asim:dim:poly:surface:0}.
\end{proof}

The following theorem addresses the braid groups of the closed disk $\mathbb{D}^2$ and the plane $\mathbb{R}^2$. Note that $\mathrm{B}_n(\mathbb{D}^2)$ is isomorphic to the classical Artin braid group $\mathrm{B}_n\cong \mathrm{B}_n(\mathbb{R}^2)$.

\begin{theorem}\label{main:thm2}
If $\Sigma$ is either $\mathbb{D}^2$ or $\mathbb{R}^2$ then $\mathrm{B}_n(\Sigma)$ is a duality group of dimension $n-1$. In particular, $\asd(\mathrm{B}_n(\Sigma))=\cd(\mathrm{B}_n(\Sigma))=n-1$. 
\end{theorem}

\begin{proof}
    We proceed by induction on $n$. For $n=1$, $\mathrm{B}_1(\Sigma)=\pi_1(\Sigma)$ is the trivial group and the claim holds. Assuming the result for $n=k$, consider the Fadell--Neuwirth short exact sequence (\cref{Fadell:Neuwirth}):
    \[
    1 \to \pi_1(\Sigma \setminus Q_k) \to \pb_{k+1}(\Sigma) \to \pb_k(\Sigma) \to 1.
    \]
    
    By the inductive hypothesis, $\pb_k(\Sigma)$ is a duality group of dimension $k-1$. Since $\pi_1(\Sigma \setminus Q_k)$  is a free group of finite rank (a duality group of dimension one), it follows from \cref{Duality:group:s.e.s} that $\pb_{k+1}(\Sigma)$ is a duality group of dimension $k$.
\end{proof}

For $\Sigma \in \{\mathbb{S}^2,\mathbb{RP}^2\}$, we show that the asymptotic dimension of $\mathrm{B}_n(\Sigma)$ matches the virtual cohomological dimension computed in \cite{maldonado:Guaschi:1}.    

\begin{theorem}\label{asd:sphere}
Let $\Sigma \in \{\mathbb{S}^2,\mathbb{RP}^2\}$ and $n \geq1$. Then $\mathrm{B}_n(\Sigma)$ is a virtual duality group and $\asd(\mathrm{B}_n(\Sigma))=\vcd(\mathrm{B}_n(\Sigma))$. Specifically, 

\begin{eqnarray*}
    \vcd(\mathrm{B}_n(\mathbb{S}^2))= \begin{cases}
    
    n-3 & \text{ if } n\geq3,\\
    
    0 & \text{ otherwise; } \\
   
    \end{cases} &   & \vcd(\mathrm{B}_n(\mathbb{RP}^2))=\begin{cases}
    
    n-2 & \text{ if } n\geq3,\\
    
    0 & \text{ otherwise.} \\
   
    \end{cases}
\end{eqnarray*}
\end{theorem}

\begin{proof}
If  \(\Sigma = \mathbb{S}^2\) (resp., $\mathbb{RP}^2$) and \(n \leq 3\) (resp., \(n \leq 2\)), $B_n(\Sigma)$ is finite and the claim is trivial.


For \(n \geq 3\) (resp., \(n \geq 2\) if $\Sigma=\mathbb{RP}^2$), we have the exact sequence \cite[Section 2.4]{MR3382024}: 
\begin{equation}\label{eq:1}
1 \to \mathbb{Z}_2 \to \mathrm{B}_n(\Sigma) \to \mcg(\Sigma\setminus Q_n) \to 1,
\end{equation}

By \cref{vcd:extension}, $\asd(\mathrm{B}_n(\Sigma))=\asd(\mcg(\Sigma-Q_n))$ and $\vcd(\mathrm{B}_n(\Sigma))=\vcd(\mcg(\Sigma\setminus Q_n)).$ For $\mathbb{RP}^2$, it is know $\vcd(\mcg(\mathbb{RP}^2 \setminus Q_3))=1$ \cite[Theorem~6.9]{I87}, which implies the group is virtually free and thus $\asd(\mcg(\mathbb{RP}^2 \setminus Q_3))=1$. A similar argument shows that $\vcd(\mcg(\mathbb{S}^2 \setminus Q_4))=1=\asd(\mcg(\mathbb{S}^2 \setminus Q_4))$. The general case follows by applying \cref{asd:mcg:Birman}.

The virtual duality property follows from \cref{eq:1} and the fact that mapping class groups are virtual duality groups \cite{harer1986, I87}. 
\end{proof}

\subsection{Cohomological dimension of braid groups of finite-type surfaces}\label{Secc:cd-SurfacesFT}
By Theorem \ref{main:thm}, surface braid groups (excluding those of the sphere and projective plane) are duality groups. In the following result, we specify their dimension. Recall that for a duality group, the cohomological dimension coincides with its dimension as a duality group (\cref{remark:cd-duality-group}).

\begin{theorem}[Surface braid groups are duality groups]\label{surfaces:braid:groups:are:duality:groups}
    Let \(\Sigma\) be a finite-type surface with infinite fundamental group and $n\geq 1$. Then \(\mathrm{B}_n(\Sigma) \) and $\pb_n(\Sigma)$ are duality groups of dimension $n+1$ if $\Sigma$ is closed, and $n$ otherwise. 
\end{theorem}

\begin{proof}
By \cref{surfacebraid:puntures:boundaries}, we need only consider surfaces with empty boundary. Furthermore, since $\mathrm{B}_n(\Sigma)$ is torsion free and $\pb_n(\Sigma)$ is a finite-index subgroup, \cref{Duality:group:extension} allows us to focus on the pure braid group.

If \(\Sigma\) has at least one puncture, we proceed by induction on \(n\). For \(n = 1\), \(\pb_1(\Sigma) = \pi_1(\Sigma)\) is a free group of positive rank, hence a duality group of dimension one. Assuming the claim for \(n\), the Fadell--Neuwirth exact sequence (\cref{Fadell:Neuwirth})
    \[
    1 \to \pi_1(\Sigma - Q_n) \to \pb_{n+1}(\Sigma) \to \pb_n(\Sigma) \to 1,
    \]
    combined with \cref{Duality:group:s.e.s} shows that \(\pb_{n+1}(\Sigma)\) is a duality group of dimension \(n + 1\)

    The case where $\Sigma$ is closed follows similarly, using the base case $n=1$ where $\pi_1(\Sigma)$ is a duality group of dimension two. 
\end{proof}

\section{Asymptotic and cohomological dimension of infinite-type surface braid groups}\label{Secc:cd-SurfacesIT}

In \cref{Secc:cd-SurfacesFT}, we computed the cohomological dimension of finite-type surface braid groups by leveraging their structure as virtual duality groups. However, braid groups of infinite-type surfaces are not duality groups, as the latter must be of type 
 and, consequently, finitely generated (see \cref{thm:DualitythenFP} and \cref{braid:groups:infinite:type:are normally:polifree}). In this section, we establish that for any infinite-type surface $\Sigma$ and every $n\geq 1$, $\asd(\pb_n(\Sigma))=\cd(\pb_n(\Sigma))$. The proof relies on our previous computations for the finite-type case and the fact that $\mathrm{PB}_n(\Sigma)$ is a normally poly-free group of length $n$.\\

Let \( M \) be a surface and \( N \subseteq M \) a subsurface such that every boundary component of \( N \) is either a boundary component of \( M \) or is contained in \( \mathrm{int}(M) \). For \( p \in N \), the inclusion \( i \colon N \hookrightarrow M \) induces a homomorphism $\psi \colon \pi_1(N, p) \to \pi_1(M, p).$ The following results extend Propositions 2.1 and 2.2 of \cite{Paris:Rolfsen} to the general topological setting.

\begin{proposition}\label{subsurface:funadamental:group}
  Let $M$ be a surface and \( N\subseteq M \) be a subsurface with \( \pi_1(N, p) \neq 1 \). The homomorphism \( \psi\) is injective if and only if no connected component of \( M \setminus N \) is a disk.   
\end{proposition}

For $n\geq 1$ and $Q_n  \in F_n(N)$, the inclusion \( N \hookrightarrow M \) induces a homomorphism \[
\psi_n \colon \mathrm{PB}_n(N) \to \mathrm{PB}_n(M).
\]
\begin{proposition}\label{subsurface:pure:braid:groups}
Let $M$ be a surface other than $\mathbb{S}^2$ and $\mathbb{RP}^2$, and  let \( N\subseteq M \) be a subsurface such that no connected component of \( M \setminus  N \) is a disk. Then $\psi_n$ is injective for all $n \geq 1$.
\end{proposition}

\begin{proof}
 The proof is by induction on \( n \). The case \( n = 1 \) follows from \cref{subsurface:funadamental:group}. Assuming the result for \( k = n - 1 \), the Fadell--Neuwirth exact sequence yields the following commutative diagram: 
\[
\begin{tikzcd}
1 \arrow[r] & \pi_1(N \setminus Q_{n-1}) \arrow[r] \arrow[d, "\varphi"] & \mathrm{PB}_n(N) \arrow[r] \arrow[d, "\psi_n"] & \mathrm{PB}_{n-1}(N) \arrow[r] \arrow[d, "\psi_{n-1}"] & 1 \\
1 \arrow[r] & \pi_1(M \setminus Q_{n-1}) \arrow[r] & \mathrm{PB}_n(M) \arrow[r] & \mathrm{PB}_{n-1}(M) \arrow[r] & 1
\end{tikzcd}
\]

The map $\varphi$ is injective by the hypothesis on $N$, and \( \psi_{n-1} \) is injective by the induction hypothesis. The injectivity of \( \psi_n \) follows from Five lemma.  
\end{proof}

\begin{theorem}\label{cd:infinite:case:brain}
    Let $\Sigma$ be an infinite-type surface and $n\geq 1$. Then $$\gd(\mathrm{PB}_n(\Sigma))=\cd(\mathrm{PB}_n(\Sigma))=n.$$
\end{theorem}

\begin{proof}
It suffices to show that $n\geq \gd(\mathrm{PB}_n(\Sigma))\geq\cd(\mathrm{PB}_n(\Sigma))\geq n.$

  For the first inequality, we use induction. For $n=1$, $\mathrm{PB}_1(\Sigma)=\pi_1(\Sigma)$ is a free group of infinite rank, so $\gd(\pb_1(\Sigma))= \cd (\pb_1(\Sigma)) =1$. Assuming the claim for all $n<k$, consider the Fadell--Neuwirth exact sequence (\cref{Fadell:Neuwirth}): 
\[
1 \to \pi_1(\Sigma \setminus Q_{k-1}) \to \pb_{k}(\Sigma) \xrightarrow{\varphi} \pb_{k-1}(\Sigma) \to 1.
\]
Since $\pb_{k-1}(\Sigma)$ is torsion-free, Theorem 5.15 in \cite{Lu05} implies  $$\gd(\pb_{k}(\Sigma))\leq \gd(\pi_1(\Sigma \setminus Q_{k-1}))+ \gd(\pb_{k-1}(\Sigma))\leq k;$$ the last inequality is by induction hypothesis and since $\pi_1(\Sigma \setminus Q_{k-1})$ is a free group. 

The inequality $\gd(\mathrm{PB}_n(\Sigma))\geq\cd(\mathrm{PB}_n(\Sigma))$ is standard. For $\cd(\mathrm{PB}_n(\Sigma))\geq n$, let $N\subseteq M$ be a finite-type subsurface with sufficiently high topological complexity such that $\psi_n$ is injective. By monotonicity $\cd$ of the cohomological dimension and \cref{surfaces:braid:groups:are:duality:groups}, $\cd(\mathrm{PB}_n(\Sigma))\geq \cd(\mathrm{PB}_n(N))=n$. 
\end{proof}

Since the pure braid group is a finite-index subgroup of the braid group and the latter has finite cohomological dimension, \cref{cd:infinite:case:brain} yields the following:

\begin{corollary}
Let $\Sigma$ be an infinite-type surface and $n \geq 1$. Then $$\cd(\mathrm{B}_n(\Sigma))=n.$$
\end{corollary}

To show that the asymptotic dimension of $\pb_n(\Sigma)$ coincides with its cohomological dimension for any infinite-type surface $\Sigma$,  we require the following structural result.

\begin{proposition}\label{braid:groups:infinite:type:are normally:polifree}
    Let $\Sigma$ be an infinite-type surface and $n \geq 1$. Then
    \begin{enumerate}[a)]
        \item  $\pb_n(\Sigma)$ is normally poly-free of length  $n$.
        \item $\pb_n(\Sigma)$ and $\mathrm{B}_n(\Sigma)$ are countable and infinitely generated.
    \end{enumerate}
\end{proposition}

\begin{proof} For item \emph{a)}, we first show by induction on $n$ that $\pb_n(\Sigma)$ is normally poly-free of length at most $n$. The case $n = 1$ is clear, as $\pb_1(\Sigma) = \pi_1(\Sigma)$ is a free group of infinite rank. Assuming the claim for all $k < n$, consider the Fadell--Neuwirth short exact sequence:
    \[
1 \to \pi_1(\Sigma \setminus Q_{n-1}) \to \pb_n(\Sigma) \xrightarrow{\varphi} \pb_{n-1}(\Sigma) \to 1.
\]

By the inductive hypothesis and \cref{normmally:polisurfaces:ses}, $\pb_n(\Sigma)$ is normally poly-free of length at most $n$. If its length were strictly less than $n$, then by \cite[Theorem 1.1(a)]{jimenezleon}, we would have $\cd(\pb_n(\Sigma)) \leq\gd(\pb_n(\Sigma)) < n$, which contradicts \cref{cd:infinite:case:brain}. Thus, $\pb_n(\Sigma)$ has length exactly $n$.

Item \emph{b)} follows by a similar induction using the same exact sequence, noting that the kernel and quotient are countable and infinitely generated at each step.
\end{proof}

\begin{remark}
Combining \cref{braid:groups:infinite:type:are normally:polifree} with the work of Jiménez--Rolland and the first author \cite[Theorem 1.1]{jimenezleon}, it follows that the virtually cyclic dimension of $\pb_n(\Sigma)$ is bounded above by $3n-1$.
\end{remark}

\begin{theorem}\label{thm:asdimInfinite-typeSurfaces}
Let $\Sigma$ be an infinite-type surface and $n \geq 1$. Then $$\asd(\pb_n(\Sigma))=\asd(\mathrm{B}_n(\Sigma))=n.$$
\end{theorem}
\begin{proof}
Since the asymptotic dimension is a quasi-isometric invariant and $\pb_n(\Sigma)$ has finite index in $\mathrm{B}_n(\Sigma)$, it suffices to show that $\asd(\pb_n(\Sigma))=n$. The upper bound $\asd(\pb_n(\Sigma)) \leq n$ follows from \cref{asim:dim:poly:free}, given that $\pb_n(\Sigma)$ is poly-free of length $n$. For the lower bound, let $N\subseteq \Sigma$ be a finite-type subsurface satisfying the condition of \cref{subsurface:pure:braid:groups}. By monotonicity of the asymptotic dimension for countable groups, we have $\asd(\pb_n(\Sigma))\geq \asd(\pb_n(N))$. The result then follows from \cref{thm:duality:braid:0} and \cref{main:thm}.
\end{proof}

\bibliographystyle{alpha} 
\bibliography{mybib}

@article{Aramayona-Vlamis2020,
  title={Big mapping class groups: an overview},
  author={Aramayona, Javier and Vlamis, Nicholas G},
  journal={in the tradition of Thurston: geometry and topology},
  pages={459--496},
  year={2020},
  publisher={Springer}
}

@article {Szepietowski2008,
    AUTHOR = {Szepietowski, B{\l}a{\.z}ej},
     TITLE = {A presentation for the mapping class group of a non-orientable
              surface from the action on the complex of curves},
   JOURNAL = {Osaka J. Math.},
  FJOURNAL = {Osaka Journal of Mathematics},
    VOLUME = {45},
      YEAR = {2008},
    NUMBER = {2},
     PAGES = {283--326},
      ISSN = {0030-6126},
   MRCLASS = {57N05 (20F38 57M50)},
  MRNUMBER = {2441942},
MRREVIEWER = {Andrew\ Putman},
       URL = {http://projecteuclid.org.pbidi.unam.mx:8080/euclid.ojm/1216151101},
}

@article{harer1986,
  title={The virtual cohomological dimension of the mapping class group of an orientable surface},
  author={Harer, John L},
  journal={Inventiones mathematicae},
  volume={84},
  number={1},
  pages={157--176},
  year={1986},
  publisher={Springer-Verlag Berlin/Heidelberg}
}

@article{FadellNeu,
 ISSN = {00255521, 19031807},
 URL = {http://www.jstor.org/stable/24489273},
 author = {E. Fadell and L. Neuwirth},
 journal = {Mathematica Scandinavica},
 pages = {111--118},
 publisher = {Mathematica Scandinavica},
 title = {CONFIGURATION SPACES},
 urldate = {2025-05-20},
 volume = {10},
 year = {1962}
}

@article{richards1963,
  title={On the classification of noncompact surfaces},
  author={Richards, I.},
  journal={Transactions of the American Mathematical Society},
  volume={106},
  number={2},
  pages={259--269},
  year={1963},
  publisher={JSTOR}
}

@book{kassel2008,
  title={Braid groups},
  author={Kassel, C. and Turaev, V.},
  volume={247},
  year={2008},
  publisher={Springer Science \& Business Media}
}

@article{YU,
 ISSN = {0003486X},
 URL = {http://www.jstor.org/stable/121011},
 author = {G. Yu},
 journal = {Annals of Mathematics},
 number = {2},
 pages = {325--355},
 publisher = {Annals of Mathematics},
 title = {The Novikov Conjecture for Groups with Finite Asymptotic Dimension},
 urldate = {2024-10-18},
 volume = {147},
 year = {1998}
}

@article {MR3606454,
    AUTHOR = {Guentner, E. and Willett, R. and Yu, G.},
     TITLE = {Dynamic asymptotic dimension: relation to dynamics, topology,
              coarse geometry, and {$C^*$}-algebras},
   JOURNAL = {Math. Ann.},
  FJOURNAL = {Mathematische Annalen},
    VOLUME = {367},
      YEAR = {2017},
    NUMBER = {1-2},
     PAGES = {785--829},
      ISSN = {0025-5831,1432-1807},
   MRCLASS = {54F45 (22A22 22F05 37A55 37B05 46L05 54H20)},
  MRNUMBER = {3606454},
MRREVIEWER = {Thomas\ Weighill},
       DOI = {10.1007/s00208-016-1395-0},
       URL = {https://doi.org/10.1007/s00208-016-1395-0},
}

@article{jimenezleon,
  author  = {R. Jim\'{e}nez Rolland and P. L. Le\'{o}n \'{A}lvarez},
  title   = {On the Virtually Cyclic Dimension of Normally Poly-Free Groups},
  journal = {Bolet\'{i}n de la Sociedad Matem\'{a}tica Mexicana},
  volume  = {31},
  number  = {5},
  year    = {2025},
  doi     = {10.1007/s40590-024-00670-z},
  url     = {https://doi.org/10.1007/s40590-024-00670-z}
}

@book {MR600654,
    AUTHOR = {Hungerford, T. W.},
     TITLE = {Algebra},
    SERIES = {Graduate Texts in Mathematics},
    VOLUME = {73},
      NOTE = {Reprint of the 1974 original},
 PUBLISHER = {Springer-Verlag, New York-Berlin},
      YEAR = {1980},
     PAGES = {xxiii+502},
      ISBN = {0-387-90518-9},
   MRCLASS = {00A05 (15-01 16-01)},
  MRNUMBER = {600654},
}

@article{WINTER2010461,
  title = {The nuclear dimension of $C^*$-algebras},
  journal = {Advances in Mathematics},
  volume = {224},
  number = {2},
  pages = {461-498},
  year = {2010},
  issn = {0001-8708},
  doi = {https://doi.org/10.1016/j.aim.2009.12.005},
  url = {https://www.sciencedirect.com/science/article/pii/S0001870809003740},
  author = {W. Winter and J. Zacharias}
}

@article{Haydee,
  author = {Contreras Peruyero, H. and Suárez-Serrato, P.},
  title = {Asymptotic dimension and geometric decompositions in dimensions 3 and 4},
  journal = {Journal of the Australian Mathematical Society},
  pages = {1--26},
  year = {2025},
  doi = {10.1017/S1446788725000072},
  note = {Published online}
}

@article{Bestvina,
     author = {Bestvina, M. and Bromberg, K. and Fujiwara, K.},
     title = {Constructing group actions on quasi-trees and applications to mapping class groups},
     journal = {Publications Math\'ematiques de l'IH\'ES},
     pages = {1--64},
     publisher = {Springer Berlin Heidelberg},
     address = {Berlin/Heidelberg},
     volume = {122},
     year = {2015},
     doi = {10.1007/s10240-014-0067-4},
     language = {en},
     url = {http://www.numdam.org/articles/10.1007/s10240-014-0067-4/}
}

@article{Eckmann1973,
  author = {Eckmann, B. and Bieri, R.},
  journal = {Inventiones mathematicae},
  pages = {103-124},
  title = {Groups with Homological Duality Generalizing Poincaré Duality.},
  url = {http://eudml.org/doc/142208},
  volume = {20},
  year = {1973},
}

@article{brown1975homological,
  title={Homological criteria for finiteness},
  author={Brown, K. S},
  journal={Comment. Math. Helv},
  volume={50},
  number={1},
  pages={129--135},
  year={1975}
}

@article{Dra:Alex,
  title={Cohomological approach to asymptotic dimension},
  author={Dranishnikov, A.},
  journal={Geometriae Dedicata},
  volume={141},
  pages={59--86},
  year={2009},
  publisher={Springer}
}

@article{Arnold2014,
  title={On some topological invariants of algebraic functions},
  author={Arnold, V.},
  journal={Trans. Moscow Math. Soc.},
  volume={21},
  pages={33--52},
  year={1970}
}

@article{Paris:Rolfsen,
     author = {Paris, L. and Rolfsen, D.},
     title = {Geometric subgroups of surface braid groups},
     journal = {Annales de l'Institut Fourier},
     pages = {417--472},
     publisher = {Association des Annales de l{\textquoteright}institut Fourier},
     volume = {49},
     number = {2},
     year = {1999},
     doi = {10.5802/aif.1680},
     zbl = {0962.20028},
     mrnumber = {2000f:20059},
     language = {en},
     url = {https://aif.centre-mersenne.org/articles/10.5802/aif.1680/}
}

@article {MR1797585,
    AUTHOR = {Aravinda, C. S. and Farrell, F. T. and Roushon, S. K.},
     TITLE = {Algebraic {$K$}-theory of pure braid groups},
   JOURNAL = {Asian J. Math.},
  FJOURNAL = {Asian Journal of Mathematics},
    VOLUME = {4},
      YEAR = {2000},
    NUMBER = {2},
     PAGES = {337--343},
      ISSN = {1093-6106},
   MRCLASS = {19B28 (19A31 20F36 20H15)},
  MRNUMBER = {1797585},
MRREVIEWER = {Ross Staffeldt},
       DOI = {10.4310/AJM.2000.v4.n2.a4},
       URL = {https://doi-org.pbidi.unam.mx:2443/10.4310/AJM.2000.v4.n2.a4},
}

@book {FM12,
    AUTHOR = {Farb, B. and Margalit, D.},
     TITLE = {A primer on mapping class groups},
    SERIES = {Princeton Mathematical Series},
    VOLUME = {49},
 PUBLISHER = {Princeton University Press, Princeton, NJ},
      YEAR = {2012},
     PAGES = {xiv+472},
      ISBN = {978-0-691-14794-9},
   MRCLASS = {57M50 (20F36 20F65 57M07 57N05)},
  MRNUMBER = {2850125},
MRREVIEWER = {Stephen P. Humphries},
}

@article {MR596323,
    AUTHOR = {Meier, D.},
     TITLE = {On the homological dimension of poly-locally free groups},
   JOURNAL = {J. London Math. Soc. (2)},
  FJOURNAL = {Journal of the London Mathematical Society. Second Series},
    VOLUME = {22},
      YEAR = {1980},
    NUMBER = {3},
     PAGES = {449--459},
      ISSN = {0024-6107},
   MRCLASS = {20J05},
  MRNUMBER = {596323},
MRREVIEWER = {U. Stammbach},
       DOI = {10.1112/jlms/s2-22.3.449},
       URL = {https://doi-org.pbidi.unam.mx:2443/10.1112/jlms/s2-22.3.449},
}

@book {Br94,
    AUTHOR = {Brown, K. S.},
     TITLE = {Cohomology of groups},
    SERIES = {Graduate Texts in Mathematics},
    VOLUME = {87},
      NOTE = {Corrected reprint of the 1982 original},
 PUBLISHER = {Springer-Verlag, New York},
      YEAR = {1994},
     PAGES = {x+306},
      ISBN = {0-387-90688-6},
   MRCLASS = {20J05 (20-02)},
  MRNUMBER = {1324339},
}

@article {MR4665278,
    AUTHOR = {Tselekidis, P.},
     TITLE = {Asymptotic dimension of graphs of groups and one-relator
              groups},
   JOURNAL = {Algebr. Geom. Topol.},
  FJOURNAL = {Algebraic \& Geometric Topology},
    VOLUME = {23},
      YEAR = {2023},
    NUMBER = {8},
     PAGES = {3587--3613},
      ISSN = {1472-2747,1472-2739},
   MRCLASS = {20E08 (20E06 20F65 20F69)},
  MRNUMBER = {4665278},
MRREVIEWER = {Sam\ Shepherd},
       DOI = {10.2140/agt.2023.23.3587},
       URL = {https://doi.org/10.2140/agt.2023.23.3587},
}

@incollection {Gr93,
    AUTHOR = {Gromov, M.},
     TITLE = {Asymptotic invariants of infinite groups},
 BOOKTITLE = {Geometric group theory, {V}ol.\ 2 ({S}ussex, 1991)},
    SERIES = {London Math. Soc. Lecture Note Ser.},
    VOLUME = {182},
     PAGES = {1--295},
 PUBLISHER = {Cambridge Univ. Press, Cambridge},
      YEAR = {1993},
      ISBN = {0-521-44680-5},
   MRCLASS = {20F32 (57M07)},
  MRNUMBER = {1253544},
}

@InCollection{Lu05,
  Title                    = {Survey on classifying spaces for families of subgroups},
  Author                   = {L{\"u}ck, W.},
  Booktitle                = {Infinite groups: geometric, combinatorial and dynamical aspects},
  Publisher                = {Birkh\"auser, Basel},
  Year                     = {2005},
  Pages                    = {269--322},
  Series                   = {Progr. Math.},
  Volume                   = {248}
}

@book {MR0375281,
    AUTHOR = {Birman, Joan S.},
     TITLE = {Braids, links, and mapping class groups},
    SERIES = {Annals of Mathematics Studies, No. 82},
 PUBLISHER = {Princeton University Press, Princeton, N.J.; University of
              Tokyo Press, Tokyo},
      YEAR = {1974},
     PAGES = {ix+228},
   MRCLASS = {55A25},
  MRNUMBER = {0375281},
MRREVIEWER = {Wilbur Whitten},
}

@article {MR3869010,
    AUTHOR = {Lima Gon\c{c}alves, D. and Guaschi, J. and Maldonado,
              M.},
     TITLE = {Embeddings and the (virtual) cohomological dimension of the
              braid and mapping class groups of surfaces},
   JOURNAL = {Confluentes Math.},
  FJOURNAL = {Confluentes Mathematici},
    VOLUME = {10},
      YEAR = {2018},
    NUMBER = {1},
     PAGES = {41--61},
   MRCLASS = {57N05 (20F36 20F38 20J06 55P20 55R80 57M07)},
  MRNUMBER = {3869010},
MRREVIEWER = {B\l a\.{z}ej Szepietowski},
       DOI = {10.5802/cml.45},
       URL = {https://doi-org.pbidi.unam.mx:2443/10.5802/cml.45},
}

@article{I87,
	doi = {10.1070/rm1987v042n03abeh001440},
	url = {https://doi.org/10.1070%2Frm1987v042n03abeh001440},
	year = 1987,
	month = {jun},
	publisher = {{IOP} Publishing},
	volume = {42},
	number = {3},
	pages = {55--107},
	author = {N V Ivanov},
	title = {Complexes of curves and the Teichmüller modular group},
	journal = {Russian Mathematical Surveys}
	}

@article{maldonado:Guaschi:1,
  title={Inclusion of configuration spaces in Cartesian products, and the virtual cohomological dimension of the braid groups of $\mathbb{S}^2$ and $\mathbb{R}\mathbb{P}^2$},
  author={Gon{\c{c}}alves, D. and Guaschi, J.},
  journal={Pacific Journal of Mathematics},
  volume={287},
  number={1},
  pages={71--99},
  year={2017},
  publisher={Mathematical Sciences Publishers}
}

@article {A2006,
    AUTHOR = {Dranishnikov, A. and Smith, J.},
     TITLE = {Asymptotic dimension of discrete groups},
   JOURNAL = {Fund. Math.},
  FJOURNAL = {Fundamenta Mathematicae},
    VOLUME = {189},
      YEAR = {2006},
    NUMBER = {1},
     PAGES = {27--34},
      ISSN = {0016-2736,1730-6329},
   MRCLASS = {20F69},
  MRNUMBER = {2213160},
       DOI = {10.4064/fm189-1-2},
       URL = {https://doi.org/10.4064/fm189-1-2},
}

@incollection {MR3382024,
    AUTHOR = {Guaschi, J. and Juan-Pineda, D.},
     TITLE = {A survey of surface braid groups and the lower algebraic
              {$K$}-theory of their group rings},
 BOOKTITLE = {Handbook of group actions. {V}ol. {II}},
    SERIES = {Adv. Lect. Math. (ALM)},
    VOLUME = {32},
     PAGES = {23--75},
 PUBLISHER = {Int. Press, Somerville, MA},
      YEAR = {2015},
   MRCLASS = {20F36 (19A31 19B28)},
  MRNUMBER = {3382024},
MRREVIEWER = {Inasa Nakamura},
}

@article{PetersenWade2024,
  author    = {Dan Petersen and Richard D. Wade},
  title     = {The handlebody group is a virtual duality group},
  journal   = {arXiv preprint arXiv:2405.15515},
  year      = {2024},
  url       = {https://arxiv.org/abs/2405.15515}
}

@article{BorelSerre1973,
  author = {Borel, Armand and Serre, Jean-Pierre},
  title = {Corners and Arithmetic Groups},
  journal = {Commentarii Mathematici Helvetici},
  volume = {48},
  pages = {436--491},
  year = {1973},
  doi = {10.1007/BF02567589}
}
\end{document}